\documentclass[11pt]{amsart}

\usepackage{xypic}  
\usepackage{amssymb}
\usepackage{amsthm}
\usepackage{epsfig}
\usepackage{paralist, comment}
\usepackage{lmodern}
\usepackage{longtable}
\usepackage{graphicx}

\usepackage{comment}

\usepackage[all,cmtip]{xy}
\usepackage{amsmath}
\usepackage{xcolor}

\newtheorem{thm}{Theorem}[section]

\newtheorem{proposition}[thm]{Proposition}

\newtheorem{claim}[thm]{Claim}

\newtheorem{conjecture}[thm]{Conjecture}

\theoremstyle{definition}

\newtheorem{remark}[thm]{Remark}

  \newtheorem{definition-remark}[thm]{Definition-Remark}

\def\geq{\geqslant}
\def\leq{\leqslant}

\def\min{\operatorname{min}}

\def\rank{\operatorname{rank}}

\def\c1{\operatorname{c_1}}
\def\c2{\operatorname{c_2}}

\def\CC{{\mathbb C}}
\def\AA{{\mathbb A}}

\def\PP{{\mathbb P}}

\def\c{\mathfrak{c}}

\def\+{\oplus}               
\def\*{\otimes}                  

\begin{document}

\title{On the infinitesimal Terracini lemma}

\author[C.~Ciliberto]{Ciro Ciliberto}
\address{Ciro Ciliberto, Dipartimento di Matematica, Universit\`a di Roma Tor Vergata, Via della Ricerca Scientifica, 00173 Roma, Italy}
\email{cilibert@mat.uniroma2.it}



 \begin{abstract}  In this paper we prove an infinitesimal version of the classical Terracini Lemma for 3--secant planes to a variety. Precisely we prove that if 
 $X\subseteq \PP^r$ is an irreducible, non--degenerate, projective complex variety of dimension $n$ with $r\geq 3n+2$, such that the variety of osculating planes to curves in $X$ has the expected dimension $3n$ and for every $0$--dimensional, curvilinear scheme $\gamma$ of length 3 contained in $X$ the family of hyperplanes sections of $X$ which are singular along $\gamma$ has dimension larger that $r-3(n+1)$, then $X$ is $2$--secant defective. 
\end{abstract}

\maketitle


\vspace{-1cm}

\section{Introduction}\label{sec:intro}

The classical Terracini Lemma describes the tangent space to the secant variety of a given variety at its general point. To be precise, given $X\subseteq \PP^r$ an irreducible, projective complex variety of dimension $n$, consider the \emph{$k$--secant variety} variety ${\rm Sec}_k(X)$ of $X$, which is the Zariski closure of the union of all $k$--dimensional linear spaces which are $(k+1)$--secant to $X$ at linearly independent points. One writes ${\rm Sec}(X)$ for ${\rm Sec}_1(X)$. One has  
$$\dim({\rm Sec}_k(X))\leqslant \min\{r, kn+n+k\}$$
and 
$$\delta_k(X):=\min\{r, kn+n+k\}-\dim({\rm Sec}_k(X))$$
is called the  \emph{$k$--defect} of $X$, and $X$ is said to be \emph {k--defective}, if $\delta_k(X)>0$. 

Given an irreducible, projective variety $X\subseteq \PP^r$ of dimension $n$ and a smooth point $p\in X$, we denote by $T_{X,p}$ the \emph{(projective) tangent space} to $X$ at $p$, which is a $n$--dimensional linear subspace of $\PP^r$.

 \begin{thm}[Terracini Lemma]\label{thm:lemma terra}
 Let $X\subseteq \PP^r$ be an irreducible, projective variety of dimension $n$. Let $p_0,\ldots, p_k$ be general linearly independent points of  $X$ and let $p\in \langle p_0,\ldots, p_k\rangle$ be a general point of  ${\rm Sec}_k(X)$. One has:
$$T_{{\rm Sec}_k(X),p}=\langle T_{X,p_0},\ldots, T_{X,p_k}\rangle$$
hence $\delta_k(X)$ is the \emph{independency defect} of the linear spaces $T_{X,p_0},\ldots, T_{X,p_k}$, i.e.,
$$\delta_k(X)= \min\{r, kn+n+k\}  - \dim (\langle T_{X,p_0},\ldots, T_{X,p_k}\rangle).$$
\end{thm}

In the paper \cite{terra42} (see also \cite {Cili}), Terracini, extending a previous result on surfaces by C. Segre \cite {Segre21}, proved what we can consider to be an \emph{infinitesimal version} of his lemma in the case $k=1$, i.e., for secant lines. Before stating Terracini's result we introduce a definition. Let $X\subseteq \PP^{r}$ be an irreducible, non--degenerate, projective variety of dimension $n$. For every non--negative integer $m$, define ${\rm Osc}_m(X)$ as the Zariski closure of the union of all the $m$--dimensional osculating spaces to smooth curves contained in the smooth locus of $X$. This is called the \emph{variety of $m$--osculating spaces} to $X$. Of course  ${\rm Osc}_0(X)=X$ and, if $X$ is smooth, then ${\rm Osc}_1(X)$ coincides with the \emph{tangential variety} ${\rm Tan}(X)$ of $X$. One has
\begin{equation}\label{eq:oscul}
\dim({\rm Osc}_m(X))\leqslant \min\{(m+1)n,r\}.
\end{equation}
We will say that $X$ is \emph{$m$--osculating regular} if equality holds in \eqref {eq:oscul}. 

This is the result proved by Terracini in \cite{terra42} (see also \cite {Cili}):

\begin{thm} [Infinitesimal Terracini Lemma for secant lines]\label{thm:terrainf}
Let $X\subseteq \PP^r$ be an irreducible, non--degenerate, projective variety of dimension $n$, which is 1--osculating regular. 
Then $X$ is 1--defective
if and only if, given a general $0$--dimensional scheme $\gamma$ of lenght 2 in $X$, the dimension of the linear system of hyperplane sections of $X$ being singular along $\gamma$ has dimension larger than $r-2(n+1)$.\end{thm}

This result suggests the following general conjecture (see \cite {Cili} and comments therein):

\begin{conjecture} [General Infinitesimal Terracini Lemma]\label{conj:terra} Let $X\subseteq \PP^r$ be an irreducible, non--degenerate, projective variety of dimension $n$, with  $r\geqslant nm+n+m$. Suppose that $X$ is $m$--osculating regular.
Then $X$ is $m$--defective if and only if given the general $0$--dimensional curvilinear scheme $\gamma$ of length $m+1$ contained in $X$, the dimension of the linear system of hyperplane sections of $X$ singular along $\gamma$ has dimension larger than $r-(m+1)(n+1)$.\end{conjecture}

The present paper is devoted to the proof of this conjecture in the case $m=2$. Our main result, i.e., Theorem \ref {thm:main}, is contained in Section \ref {sec:terra}. Sections \ref {sec:1}, \ref {sec:qacurves}, \ref {sec:osc} are devoted to preliminaries and preparatory results. 

Conjecture \ref {conj:terra} for all $m$ looks quite plausible since the obstruction to proving it in its general form is not theoretical but purely technical, i.e., the necessary computations become extremely complicated and difficult to handle.
 
\medskip

\noindent {\bf Aknowledgements:} The author is a member of GNSAGA of INdAM. He acknowledges the MIUR Excellence Department Project awarded to the Department of Mathematics, University of Rome Tor Vergata, CUP E83C18000100006.

\section{The tangent space to a variety along a curvilinear scheme}\label {sec:1} 

Let $X\subseteq \PP^r$ be an irreducible, projective, non--degenerate, variety of dimension $n$. Let $\gamma$ be a $0$--dimensional, length $k$, curvilinear scheme supported at a smooth point $p\in X$. Then $\gamma$ determines a sequence $p_1, p_2,\ldots, p_{k-1}$ of \emph{infinitely near points} to $p$ on $X$, where $p_1$ is the infinitely near point to $p$ along $\gamma$, i.e., $p_1$ lies on the blow--up $X_1$ of $X$ at $p$ and it is the support of the strict transform $\gamma_1$ of $\gamma$ on $X_1$ ($\gamma_1$ is a $0$--dimensional, length $k-1$, curvilinear scheme), and inductively $p_i$ lies on the blow--up $X_i$ of $X_{i-1}$ at $p_{i-1}$ and is the support of the strict transform $\gamma_i$ of $\gamma_{i-1}$ on $X_i$ ($\gamma_i$ is a $0$--dimensional, length $k-i$, curvilinear scheme), for $2\leq i\leq k-1$. 

Given a hyperplane $\pi$ of $\PP^r$, we denote by $X_\pi$ the intersection scheme of $X$ with $\pi$.  In the above set up, we define the \emph{tangent space} $T_{X,\gamma}$ to $X$ along $\gamma$ to be the intersection of all the hyperplanes $\pi$ of $\PP^r$, such that $X_\pi$ is singular along $\gamma$, i.e.,
$\pi$ is tangent to $X$ at $p$ and moreover the strict transform of $X_\pi$ to the blow--up $X_i$ is singular at the point $p_i$, for $1\leq i\leq k-1$. If $k=1$, then $\gamma$ equals its support $p$ and $T_{X,\gamma}=T_{X,p}$. In general the expected dimension of  $T_{X,\gamma}$ is
\[
\tau_{n,k}=\min\{r, k(n+1)-1\}
\]
and 
\begin{equation}\label{eq:tan}
t_{X,\gamma}:=\dim(T_{X,\gamma})\leq \tau_{n,k}
\end{equation}
We will say that $X$ is \emph{regular} along $\gamma$ if the equality holds in \eqref {eq:tan}, and \emph{special} otherwise. There is never speciality if $k=1$, but speciality may occur for $k>1$. As mentioned in the Introduction, the case $k=2$ has been considered by Terracini in \cite{terra42} (see \cite {Cili} for comments). In the present paper we focus on the case $k=3$ and we want to characterize varieties $X$ special along a general $0$--dimensional curvilinear length 3 scheme. We will  assume $r\geq 3n+2$, so that $\tau_{n,3}=3n+2$. 

We want to give an explicit description of $T_{X,\gamma}$ with $\gamma$ a $0$--dimensional curvilinear length 3 scheme. To do so, we let $p\in X$ be a smooth point, and consider a \emph{(projective) chart} of $X$ around $p$, i.e., a local biholomorphic parametrization ${\bf x}(u_1,\ldots,u_n)$ of the smooth locus of $X$ around $p$, where ${\bf x}(u_1,\ldots,u_n)$ is a vector of homogeneous coordinates of a point on $X$ and $(u_1,\ldots,u_n)$ varies in an $n$--dimensional polidisk $\mathbb D^n$. The point $p$ corresponds to $(u_1,\ldots, u_n)={\bf 0}$ and the Jacobian of ${\bf x}(u_1,\ldots,u_n)$ with respect to $u_1,\ldots, u_n$ has  rank $n$. By expanding in Taylor series, we have
\[
{\bf x}(u_1,\ldots,u_n)=\sum_{h=0}^\infty \frac 1 {h!} {\bf x}^{(h)}(u_1,\ldots, u_n),
\]
where
\[
{\bf x}^{(h)}(u_1,\ldots, u_n)=\sum_{i_1,\ldots,i_h} {\bf x}_{i_1,\ldots,i_h}({\bf 0})u_{i_1}\cdots u_{i_h}
\]
where $(i_1,\ldots,i_h)$ varies among all choices of $h$  
indices with repetitions among $1,\ldots, n$, and 
\[
{\bf x}_{i_1,\ldots,i_h}=\frac {\partial^h {\bf x}}{\partial u_{i_1}\ldots \partial u_{i_h}}.
\]
To ease notation, if no confusion arises, we will denote ${\bf x}_{i_1,\ldots,i_h}({\bf 0})$ simply by 
 ${\bf x}_{i_1,\ldots,i_h}$. Let ${\bf  X}=(x_0:\ldots: x_r)$ be a vector of homogeneous coordinates on $\PP^r$, and consider a hyperplane $\pi$ with equation ${\bf a} \cdot {\bf  X}=0$, where ${\bf a}=(a_0,\ldots, a_r)$ is a non--zero vector. Then the equation of $X_\pi$  in the above chart is
\[
{\bf a} \cdot {\bf x}(u_1,\ldots,u_n)=0, \quad \text{i.e.,}\quad \sum_{h=0}^\infty \frac 1 {h!} {\bf a} \cdot  {\bf x}^{(h)}(u_1,\ldots, u_n)=0.
\]
Then $\pi$ is \emph {tangent} to $X$ at $p$ if and only if 
\[
{\bf a} \cdot {\bf x}={\bf a} \cdot {\bf x}_1=\ldots ={\bf a} \cdot {\bf x}_n=0,
\] 
or, in other terms, $T_{X,p}$ is generated by the points with homogenous coordinates ${\bf x}, {\bf x}_1,\ldots, {\bf x}_n$, that are linearly independent. By abusing notation we will write
\[
T_{X,p}=\langle {\bf x}, {\bf x}_1,\ldots, {\bf x}_n\rangle.
\]

Let us now consider $\gamma$ a $0$--dimensional curvilinear length 3 scheme supported at $p$. This is determined by a \emph{second order jet} parametrically given by equations of the form
\begin{equation}\label{eq:jet}
u_i=\lambda _it+ \mu_it^2, \quad \text{with} \quad i=1,\ldots, n,
\end{equation} 
where $t\in \CC[t]/(t^3)$, and one has the vectors of constants ${\bf \lambda}=(\lambda_1,\ldots, \lambda_n)\neq {\bf 0}$ and ${\bf \mu}=(\mu_1,\ldots, \mu_n)$. Up to a change of homogeneous coordinates in $\PP^r$ we may assume that ${\bf \lambda}=(1,0,\ldots, 0)$ and, with a change of the parameter $t$, we may assume, in addition,  that  $\mu_1=0$. 

We want to impose that $X_\pi$ is singular along $\Gamma$. For this we first impose that the proper transform of $X_\pi$ has a singular point at the infinitely near point $p_1$. So we blow--up $X$ at $p$, or, by working in the chart, we blow--up the polidisk $\mathbb D$ at ${\bf 0}$. The blow--up sits in $ \mathbb D\times \PP^{n-1}$ and has equations
\[
\rank \left(\begin{matrix} 
u_1& \ldots &u_{n} \\
v_1&\ldots &v_n \\
\end{matrix}\right)<2,
\]
where $(v_1: \ldots: v_n)$ are homogeneous coordinates in $\PP^{n-1}$. We put ourselves in the open subset where $v_1\neq 0$, so that we can set $v_1=1$, and the blow--up has then equations
\begin{equation}\label{eq:bu}
u_i=u_1v_i, \quad \text {for}\quad 2\leq i\leq n,
\end{equation} 
in $\mathbb D\times \mathbb A^{n-1}$. Hence in this open subset the blow--up is isomorphic to $\mathbb A^n$ with coordinates $(u_1,v_2,\ldots, v_n)$. We set $u_1:=u$. Then the  exceptional divisor has equation $u=0$ in this chart. The proper transform of the jet \eqref {eq:jet} on the blow--up has equations
\[
u=t, v_i=\mu_it, \quad \text {for}\quad 2\leq i\leq n.
\]
Setting $t=0$ we see that the infinitely near point $p_1$ in this chart is again the origin $\bf 0$. The proper transform of $X_\pi$ on the blow--up has equation
\[
\sum_{h=2}^\infty \frac 1 {h!} u^{h-2} {\bf a} \cdot  {\bf x}^{(h)}(1,v_2\ldots, v_n)=0,
\]
and we must impose that this has a singular point at the origin. Passing through the origin means that 
\begin{equation}\label{eq:a11}
{\bf a} \cdot  {\bf x}^{(2)}(1,0\ldots, 0)=0, \quad 
\text {i.e.,}\quad {\bf a} \cdot {\bf x}_{11}=0.
\end{equation}
Then we must impose the vanishing at the origin of the derivatives with respect to $u, v_2,\ldots, v_n$. By imposing that the derivative with respect to $u$ vanishes one gets
\begin{equation}\label{eq:a111}
{\bf a} \cdot  {\bf x}^{(3)}(1,0\ldots, 0)=0, \quad 
\text {i.e.,}\quad {\bf a} \cdot {\bf x}_{111}=0.
\end{equation}
By imposing that the derivative with respect to $v_i$ vanishes, for $2\leq i\leq n$, one gets
\begin{equation}\label{eq:aij}
{\bf a} \cdot  \frac \partial {\partial v_i} {\bf x}^{(2)}(1,0\ldots, 0)=0, \quad 
\text {i.e.,}\quad {\bf a} \cdot {\bf x}_{1i}=0.
\end{equation}
This means that if we denote by $\gamma'$ the $0$--dimensional scheme of length 2 which is the truncation of $\gamma$ at second order, then
\[
T_{X,\gamma'}=\langle {\bf x}, {\bf x}_1,\ldots, {\bf x}_n, {\bf x}_{11}, \ldots, {\bf x}_{1n}, {\bf x}_{111}\rangle,
\]
though now the points could be no longer linearly independent. 
 
Next we want to impose that the proper transform of $X_\pi$ has a singular point at the infinitely near point $p_2$. For this we blow--up again at $p_1$, namely we have to blow up $\mathbb A^n$ with coordinates $(u,v_1,\ldots, v_n)$ at the origin. The resulting blow--up lives in $\mathbb A^n\times \PP^{n-1}$ and has equations
\[
\rank \left(\begin{matrix} 
u& v_1&\ldots &v_{n} \\
w_1&w_2&\ldots &v_n \\
\end{matrix}\right)<2,
\]
where $(w_1: \ldots: w_n)$ are homogeneous coordinates in $\PP^{n-1}$. We put ourselves in the open subset where $w_1\neq 0$, so that we can set $w_1=1$, and the blow--up has then equations
\begin{equation}\label{eq:bu}
v_i=uw_i, \quad \text {for}\quad 2\leq i\leq n,
\end{equation} 
in $\AA^n\times \mathbb A^{n-1}$.  In this open subset the blow--up is isomorphic to $\mathbb A^n$ with coordinates $(u,w_2,\ldots, w_n)$. The exceptional divisor has again equation $u=0$. 

To see what are the coordinates of the infinitely near point $p_2$ in this chart, we take the strict transform of $\gamma$ on this blow--up. It has equations
\[
u=t, w_i=\mu_i, \quad  \text {for}\quad 2\leq i\leq n,
\]
hence, setting $t=0$ we see that $p_2$ has coordinates $(0,\mu_2,\ldots, \mu_n)$. 

Next we take the strict transform of $X_\pi$. The total transform has equation
\begin{equation}\label{eq:buq}
\sum_{h=2}^\infty \frac 1 {h!} u^{h-2} {\bf a} \cdot  {\bf x}^{(h)}(1,uw_2\ldots, uw_n)=0.
\end{equation} 
Now, notice that
\[
{\bf a} \cdot  {\bf x}^{(2)}(u_1,u_2\ldots, u_n)=\sum_{ij}a_{ij}u_iu_j, \quad  \text {were we set}\quad a_{ij}:= {\bf a}\cdot {\bf x}_{ij},
\]
and we already imposed $a_{1j}=0$ for $1\leq j\leq n$ (see \eqref {eq:a11} and \eqref {eq:aij}). Hence 
\[
{\bf a} \cdot  {\bf x}^{(2)}(1,uw_2\ldots, uw_n)=u^2\sum_{2\leq i,j\leq n}a_{ij}w_iw_j.
\]
Moreover 
\[
{\bf a} \cdot  {\bf x}^{(3)}(u_1,u_2\ldots, u_n)=\sum_{ijk}a_{ij}u_iu_ju_k, \quad  \text {were we set}\quad a_{ijk}:= {\bf a}\cdot {\bf x}_{ijk},
\]
and we have  $a_{111}=0$ by \eqref {eq:a111}. Thus ${\bf a} \cdot  {\bf x}^{(3)}(1,v_2\ldots, v_n)$ is a degree 3 polynomial in $v_2,\ldots, v_n$ with no constant term. So we can write
\[
{\bf a} \cdot  {\bf x}^{(3)}(1,v_2\ldots, v_n)=\sum_{i=1}^3\psi_i(v_2\ldots, v_n),
\]
where $\psi_i(v_2\ldots, v_n)$ is a homogeneous polynomial of degree $i$, for $1\leq i\leq 3$. Therefore
\[
{\bf a} \cdot  {\bf x}^{(3)}(1,uw_2\ldots, uw_n)=u\sum_{i=1}^3u^{i-1}\psi_i(w_2\ldots, w_n).
\]
Plugging into \eqref {eq:buq}, we see that the left hand side of \eqref {eq:buq} is divisible by $u^2$ (which was a priori clear because $p_1$ is a singular point of the proper transform of $X_\pi$). Dividing by $u^2$ we get the equation of the strict transform of $X_\pi$, namely
\[
\frac 1 {2!} \sum_{2\leq i,j\leq n}a_{ij}w_iw_j+ \frac 1{3!} \sum_{i=1}^3u^{i-1}\psi_i(w_2\ldots, w_n)+ \sum_{h=4}^\infty \frac 1 {h!} u^{h-4} {\bf a} \cdot  {\bf x}^{(h)}(1,uw_2\ldots, uw_n)=0
\]

Now we have to impose that this divisor has a singular point at $(0, \mu_2,\ldots, \mu_n)$. First of all, imposing vanishing at this point, reads
\begin{equation}\label{eq:lop}
12 \sum_{2\leq i,j\leq n}a_{ij}\mu_i\mu_j+ 4 \psi_1(\mu_2\ldots, \mu_n)+  {\bf a} \cdot  {\bf x}_{1111}=0.
\end{equation}
On the other hand one has
\[
\psi_1(\mu_2\ldots, \mu_n)=3\sum_{i=2}^na_{11i}\mu_i
\]
so that \eqref {eq:lop} reads
\[
 {\bf a} \cdot \big (12 \sum_{2\leq i,j\leq n}{\bf x}_{ij}\mu_i\mu_j+ 12 \sum_{i=2}^n{\bf x}_{11i}\mu_i +  {\bf x}_{1111}\big )=0.
\]
Next we have to impose the vanishing of the derivatives. By imposing the vanishing of the $u$--derivative and arguing as we did above, we find
\[
{\bf a} \cdot \big (60 \sum_{2\leq i,j\leq n}{\bf x}_{1ij}\mu_i\mu_j+ 20 \sum_{i=2}^n{\bf x}_{111i}\mu_i +  {\bf x}_{11111}\big )=0.
\]
Finally, imposing the vanishing of the $w_h$--derivative, with $2\leq h\leq n$, we find
\[
{\bf a} \cdot \big (2\sum_{2\leq i\leq n}{\bf x}_{ih}\mu_i+ {\bf x}_{11h})=0, \quad \text{for}\quad 2\leq h\leq n.
\]

In conclusion we can state the following:

\begin{proposition}\label{prop:tan} Let $X\subseteq \PP^r$ be an irreducible, projective, non--degenerate variety of dimension  $n$, with $r\geq 3n+2$. Let $p\in X$ be a smooth point and let ${\bf x}(u_1,\ldots,u_n)$ be a chart of $X$ around $p$. Let $\gamma$ be the curvilinear $0$--dimensional, length 3 scheme determined by the jet with equations \eqref {eq:jet}, where  $(\lambda_1,\ldots, \lambda_n)=(1,0, \ldots, 0)$ and $\mu_1=0$. Then
\begin{equation}\label{eq:part}
\begin{split}
&T_{X,\gamma}=\langle {\bf x}, {\bf x}_1,\ldots, {\bf x}_n, {\bf x}_{11}, \ldots, {\bf x}_{1n}, {\bf x}_{111},2 \sum_{2\leq i\leq n}{\bf x}_{i2}\mu_i+ {\bf x}_{112},\ldots, 2\sum_{2\leq i\leq n}{\bf x}_{in}\mu_i+ {\bf x}_{11n},\\
&12 \sum_{2\leq i,j\leq n}{\bf x}_{ij}\mu_i\mu_j+ 12 \sum_{i=2}^n{\bf x}_{11i}\mu_i +  {\bf x}_{1111}, 60 \sum_{2\leq i,j\leq n}{\bf x}_{1ij}\mu_i\mu_j+ 20 \sum_{i=2}^n{\bf x}_{111i}\mu_i +  {\bf x}_{11111}\rangle
\end{split}
\end{equation}

\end{proposition}

\section{Quasi--asymptotic curves}\label{sec:qacurves}

Let $X\subseteq \PP^r$ be an irreducible, projective, non--degenerate variety of dimension $n$ and let $p$ be a smooth point of $X$. As we did before, consider ${\bf x}(u_1,\ldots,u_n)$  a chart of $X$ around $p$. 

Given a positive integer $h$ we denote by $T^ {(h)}_{X,p}$ the \emph{$h$--osculating space} to $X$ at $p$, which is the subspace of $\PP^r$ which is spanned by the points which have the vectors of the homogeneous coordinates given by the derivatives of order at most $h$ of ${\bf x}$ at $\bf 0$, provided such vectors are non--zero. 

With the same notation as above, $T^ {(h)}_{X,p}$ is spanned by the points which have (non--zero) coordinate vectors  ${\bf x}_{j_0,\ldots, j_r}$, with $ j_0+\ldots+ j_r\leq h$. 
In particular the tangent space  $T_{X,p}$ is  $T^ {(1)}_{X,p}$. 

The osculating spaces can be defined intrinsecally, i.e., without making use of a local parametrizations of $X$. Let $\mathcal P^h_X(1)$ the \emph{sheaf of principal parts of order $h$} of $\mathcal O_X(1)$, which is a locally free sheaf of rank ${h+n}\choose n$ on the smooth locus of $X$ (cfr. \cite [Example 2.5.6]{Fulton}). There is a natural evaluation morphism
$$\rho_X^h: \mathcal O_X^{r+1}\to \mathcal P^h_X(1).$$
Then $T^ {(h)}_{X,p}$  is defined as $\mathbb P({\rm Im}(\rho_X^h(p)))$. 

Next we introduce the definition of \emph{quasi--asymtotic curves} on a variety. This concept has been introduced by E. Bompiani in \cite {Bomp13}.

Let $X\subseteq \PP^r$ be a variety as above. Consider a smooth curve $C\subseteq X$ contained in the smooth locus of $X$. One says that $C$ is a \emph{$(h,k)$--asymptotic curve} or simply a  {$\gamma_{h,k}$} of $X$, with $h<k$, if for $p\in C$ general one has
$${\dim \big (  \langle T^ {(h)}_{X,p}, T^ {(k)}_{C,p}\rangle \big)< \min \big \{r, \dim \big ( T^ {(h)}_{X,p}\big) +k-h\big\}}.$$

For example a $\gamma_{1,2}$ of $X$, is an asymptotic curve in the sense of elementary differential geometry, i.e., a curve such that for every point $p\in C$ (smooth for $C$ and $X$), the osculating plane $T^ {(2)}_{C,p}$ to $C$ at $p$ is contained in the tangent space $T_{X,p}$ to $X$ at $p$.

More generally a $\gamma_{1,k}$ of $X$ is such that for any point $p$ smooth for both $C$ and $X$, one has
$${\dim \big ( T^ {(k)}_{C,p} \cap T_{X,p}\big)\geqslant 2}.$$

It is important for our purposes to determine necessary and sufficient conditions for the existence of $\gamma_{1,5}$'s on $X$. We will assume from now on $r=3n+2$ (if $r>3n+2$ we can consider a general projection of $X$ in $\PP^{3n+2}$) and, more precisely, we want to determine necessary and sufficient conditions for the existence of a family of dimension $3(n-1)$ of $\gamma_{1,5}$'s on $X$ with the property that given a general curvilinear, 0--dimensional scheme $\gamma$ of lenght 3 on $X$, there is a unique $\gamma_{1,5}$ on $X$ containing  $\gamma$. 

We suppose, as before, given a chart of $X$ around $p\in X$ and that an analytic curve $C$ through $p$ is given via a parametric representation of the form
\begin{equation}\label{eq:funct}
u_i=u_i(t)=\lambda_i t+ \mu_i t^ 2+\nu_it^3+\rho_it^4 +\sigma_it^5+ \ldots \,\,\, ,\,\,\, \text {with} \,\,\, i=1,\ldots, n,
\end{equation}
with $t$ a parameter varying in a disc $\mathbb D$ with center the origin of $\mathbb C$, and  $(\lambda_1,\ldots, \lambda_n)\neq {\bf 0}$. We let 
\[
{\bf x}(t)={\bf x}(u_1(t),\ldots, u_n(t))
\]
be the parametrization of $C$ around $p$. We denote by ${\bf x}', {\bf x}''$ etc., the derivatives of ${\bf x}(t)$ at $t=0$. Thus we have
\[
\begin{split}
&{\bf x}'=\sum_{i=1}^n {\bf x}_i\lambda_i\\
&{\bf x}''=\sum_{i,j=1}^n {\bf x}_{ij}\lambda_i\lambda_j+2\sum_{i=1}^n{\bf x}_i\mu_i\\
&{\bf x}'''=\sum_{i,j,k=1}^n {\bf x}_{ijk}\lambda_i\lambda_j\lambda_k+6\sum_{i,j=1}^n {\bf x}_{ij}\lambda_i\mu_j+6\sum_{i=1}^n {\bf x}_i\nu_i\\
&{\bf x}''''=\sum_{i,j,k,l=1}^n {\bf x}_{ijkl}\lambda_i\lambda_j\lambda_k\lambda_l+12\sum_{i,j,k=1}^n {\bf x}_{ijk}\lambda_i\lambda_j\mu_k+12 \sum_{i,j=1}^n {\bf x}_{ij}\mu_i\mu_j+24\sum_{i,j=1}^n {\bf x}_{ij}\lambda_i\nu_j+24\sum_{i=1}^n {\bf x}_i\rho_i\\
&{\bf x}'''''=\sum_{i,j,k,l,m=1}^n {\bf x}_{ijkl}\lambda_i\lambda_j\lambda_k\lambda_l\lambda_m+20\sum_{i,j,k,l=1}^n {\bf x}_{ijkl}\lambda_i\lambda_j\lambda_k\mu_l+60\sum_{i,j,k=1}^n {\bf x}_{ijk}\lambda_i\mu_j\mu_k+\\
&+120\sum_{i,j=1}^n {\bf x}_{ij}\mu_i\nu_j+120\sum_{i,j=1}^n {\bf x}_{ij}\lambda_i\rho_j+60\sum_{i,j,k=1}^n {\bf x}_{ijk}\lambda_i\lambda_j\nu_k+120\sum_{i=1}^n {\bf x}_i\sigma_i
\end{split}
\]
and $C$ verifies the $\gamma_{1,5}$ condition at $p$ if and only if the vectors
 \[
{\bf x}, {\bf x}_1, \ldots, {\bf x}_n, {\bf x}',{\bf x}'', {\bf x}''', {\bf x}'''', {\bf x}'''''
\] 
are linearly dependent. This is equivalent to say that the vectors ${\bf x}, {\bf x}_1, \ldots, {\bf x}_n$ and
\[
\begin{split}
&\sum_{i,j=1}^n {\bf x}_{ij}\lambda_i\lambda_j\\
&\sum_{i,j,k=1}^n {\bf x}_{ijk}\lambda_i\lambda_j\lambda_k+6\sum_{i,j=1}^n {\bf x}_{ij}\lambda_i\mu_j\\
&\sum_{i,j,k,l=1}^n {\bf x}_{ijkl}\lambda_i\lambda_j\lambda_k\lambda_l+12\sum_{i,j,k=1}^n {\bf x}_{ijk}\lambda_i\lambda_j\mu_k+12 \sum_{i,j=1}^n {\bf x}_{ij}\mu_i\mu_j+24\sum_{i,j=1}^n {\bf x}_{ij}\lambda_i\nu_j\\
&\sum_{i,j,k,l,m=1}^n {\bf x}_{ijkl}\lambda_i\lambda_j\lambda_k\lambda_l\lambda_m+20\sum_{i,j,k,l=1}^n {\bf x}_{ijkl}\lambda_i\lambda_j\lambda_k\mu_l+60\sum_{i,j,k=1}^n {\bf x}_{ijk}\lambda_i\mu_j\mu_k+\\
&+120\sum_{i,j=1}^n {\bf x}_{ij}\mu_i\nu_j+120\sum_{i,j=1}^n {\bf x}_{ij}\lambda_i\rho_j+60\sum_{i,j,k=1}^n {\bf x}_{ijk}\lambda_i\lambda_j\nu_k
\end{split}
\]
are linearly dependent. Note  that we want this condition to hold given $(\lambda_1,\ldots, \lambda_n)\neq {\bf 0}$ and $(\mu_1,\ldots, \mu_n)$ arbitrarily, and for suitable $(\nu_1,\ldots, \nu_n)$, $(\rho_1,\ldots, \rho_n)$. This implies that the given condition is equivalent to the linear dependence of the $3n+3$ vectors ${\bf x}, {\bf x}_1, \ldots, {\bf x}_n$ and
\begin{equation}\label{eq:vectors}
\begin{split}
&\sum_{i=1}^n {\bf x}_{ij}\lambda_i, \quad \text{for}\quad  1\leq j\leq n\\
&\sum_{i,j,k,l=1}^n {\bf x}_{ijkl}\lambda_i\lambda_j\lambda_k\lambda_l+12\sum_{i,j,k=1}^n {\bf x}_{ijk}\lambda_i\lambda_j\mu_k+12 \sum_{i,j=1}^n {\bf x}_{ij}\mu_i\mu_j\\
&2\sum_{i=1}^n {\bf x}_{ik}\mu_i+\sum_{i,j=1}^n {\bf x}_{ijk}\lambda_i\lambda_j, \quad \text{for}\quad  1\leq k\leq n\\
&\sum_{i,j,k,l,m=1}^n {\bf x}_{ijklm}\lambda_i\lambda_j\lambda_k\lambda_l\lambda_m+20\sum_{i,j,k,l=1}^n {\bf x}_{ijkl}\lambda_i\lambda_j\lambda_k\mu_l+60\sum_{i,j,k=1}^n {\bf x}_{ijk}\lambda_i\mu_j\mu_k.
\end{split}
\end{equation}
Note that we suppressed the vector $\sum_{i,j,k=1}^n {\bf x}_{ijk}\lambda_i\lambda_j\lambda_k$ because it is linearly dependent from the vectors on the first and on the third line. 

Since we are in $\PP^{3n+2}$, the vectors in question have length $3n+3$, thus the above linear dependence condition is equivalent to the vanishing of the determinant having as columns the vectors in question. This way we have a single equation which may be interpreted as a second order differential equation in the variables $u_i(t)$, for $1\leq i\leq n$, which in turn determine the $\gamma_{1,5}$ quasi--asymptotic curves $C$ solutions of the problem. From this equation we can derive more equations for example by differentiating with respect to $\lambda_1,\ldots, \lambda_n$. Eventually we have a system of differential equations which, by the existence and uniqueness theorem of solution of differential equations, has a unique solution once one fixes the initial conditions, which consist in fixing the 0--dimensional, curvilinear scheme $\gamma$ of degree 3 such that the solution curve $C$ contains $\gamma$. In conclusion we find the following result:

\begin{thm}\label{thm:gamma} Let $X\subseteq \PP^{3n+2}$ be an irreducible, non--degenerate, projective variety of dimension $n$. Consider a chart ${\bf x}(u_1,\ldots, u_n)$ for an open subset $U$ of the smooth locus of $X$. Then a necessary and sufficient condition for the local existence in $U$ of a family of dimension $3(n-1)$ of $\gamma_{1,5}$'s with the property that given a general curvilinear, 0--dimensional scheme $\gamma$ of lenght 3 on $X$, there is a unique $\gamma_{1,5}$ on $X$ containing  $\gamma$, is the identical linear dependence of the vectors ${\bf x}, {\bf x}_1, \ldots, {\bf x}_n$ and the vectors in \eqref {eq:vectors}, where $(\lambda_1,\ldots, \lambda_n)\neq {\bf 0}$ and $(\mu_1,\ldots, \mu_n)$ stand for the first and second derivatives of the functions $u_i(t)$, $1\leq i\leq n$, as in \eqref{eq:funct}, giving a local parametrization of a $\gamma_{1,5}$ on $X$.
\end{thm}

We stress that the above theorem is local in nature, i.e., the family of dimension $3(n-1)$ of $\gamma_{1,5}$'s on $X$ exists locally in the smooth locus of $X$ around a given smooth point of $X$. Moreover, the quasi--asymptotic curves $\gamma_{1,5}$ of the family are analytic and in general not algebraic.

Next we finish this section with the following proposition:

\begin{proposition}\label{prop:equal} Let $X\subseteq \PP^{3n+2}$ be an irreducible, non--degenerate, projective variety of dimension $n$. The following propositions are equivalent:\\\begin{inparaenum} 
\item [(i)] $X$ possesses (locally) a family of dimension $3(n-1)$ of $\gamma_{1,5}$'s with the property that given a general curvilinear, 0--dimensional scheme $\gamma$ of lenght 3 on $X$, there is a unique $\gamma_{1,5}$ of the family containing  $\gamma$;\\
 \item [(ii)] for a general curvilinear, 0--dimensional scheme $\gamma$ of lenght 3 on $X$, $X$ is special along $\gamma$.
 \end{inparaenum}
\end{proposition}  

\begin {proof} Fix a general point $p\in X$ and a chart ${\bf x}(u_1,\ldots,u_n)$ of $X$ at $p$. Consider $\gamma$ a $0$--dimensional curvilinear length 3 scheme supported at $p$, determined by a {second order jet} given by equations \eqref {eq:jet}. Then Proposition \ref {prop:tan} tells us that $T_{X,\gamma}$ is as in \eqref {eq:part}. Hence $T_{X,\gamma}$ is special if and only if the vectors appearing in \eqref {eq:part} are linearly dependent. But this is exactly the same as the conditions that the vectors ${\bf x}, {\bf x}_1, \ldots, {\bf x}_n$ and the vectors appearing in \eqref {eq:vectors} are linearly dependent. This proves the assertion. 
\end{proof}

\section{Conditions for $2$--osculating regularity}\label{sec:osc}

Let $X\subseteq \PP^{r}$ be an irreducible, non--degenerate, projective variety of dimension $n$ and recall the definition of the variety of $m$--osculating spaces to $X$ from the Introduction.

We focus on the case $m=2$ and $r\geq 3n$ and we want to give conditions for the $2$--osculating regularity of $X$. As usual, consider a chart ${\bf x}(u_1,\ldots,u_n)$ of $X$ around a general point of $X$. Then a parametrization of an open subset of ${\rm Osc}_2(X)$ is given by
\[
{\bf y}(u_1,\ldots, u_n, \lambda_2, \ldots, \lambda_n,\mu_2, \ldots, \mu_n, \alpha, \beta)=
{\bf x}+\alpha\sum_{i=1}^n{\bf x}_i\lambda_i+\beta \big (\sum_{i,j=1}^n  {\bf x}_{ij}\lambda_i\lambda_j+2\sum_{i=1}^n\mu_i{\bf x}_i \big)
\]
where we set $\lambda_1=1$ and $\mu_1=0$ and ${\bf x}$ and its derivatives are functions of $u_1,\ldots, u_n$. The conditions for the $2$--osculating regularity of $X$ consist in asking that ${\bf y}$ and its derivatives with respect to $u_1,\ldots, u_n, \lambda_2, \ldots, \lambda_n,\mu_2, \ldots, \mu_n, \alpha, \beta$ are linearly independent for general values of the variables. After a standard computation one arrives at the following:

\begin{proposition}\label{prop:reg} Let $X\subseteq \PP^{r}$ be an irreducible, non--degenerate, projective variety of dimension $n$, with $r\geq 3n$. Consider a chart ${\bf x}(u_1,\ldots,u_n)$ of $X$ around a general point. 
Then $X$ is $2$--osculating regular if and only if for general values of the variables $u_1,\ldots, u_n, \lambda_2, \ldots, \lambda_n,\mu_2, \ldots, \mu_n$ the $3n+1$ points with homogeneous coordinates given by the following vectors
\begin{equation}\label{eq:vec}
\begin{split}
& {\bf x}, {\bf x}_1, \ldots, {\bf x}_n,\\
&\sum_{i=1}^n  {\bf x}_{ij}\lambda_i, \quad \text{for}\quad 1\leq j\leq n,\\
& \sum_{i,j=1}^n  {\bf x}_{kij}\lambda_i\lambda_j+2 \sum_{i=2}^n {\bf x}_{kj}\mu_j,\quad \text{for}\quad 1\leq k\leq n.
\end{split}
\end{equation}
are linearly independent.
\end{proposition}

\begin{remark}\label{rem:lap} Recall that we can always assume that a general $0$--dimensional curvilinear, scheme of degree 3 is determined by a jet with equation \eqref {eq:jet} and 
$\lambda_1=1, \lambda_i=0$ for $2\leq i\leq n$. In this case the vectors in \eqref {eq:vec} take a simpler form. 

Let us assume that the jet in question belongs to the coordinate curve with equations $u_1=t, u_i=0$ for $2\leq i\leq n$, i.e., in addition $\mu_i=0$ for $1\leq i\leq n$. Then the vectors in \eqref {eq:vec}  are 
\[
 {\bf x}, {\bf x}_1, \ldots, {\bf x}_n,
{\bf x}_{11}, \ldots, {\bf x}_{1n},
{\bf x}_{111}, \ldots, {\bf x}_{11n}
\]
and if they are independent then  $X$ is $2$--osculating regular. 
\end{remark}

\section{The infinitesimal Terracini lemma}\label{sec:terra}

In this section we prove our main result, i.e., the theorem:

\begin{thm}\label{thm:main} Let $X\subseteq \PP^r$ be an irreducible, non--degenerate, projective variety of dimension $n$ with $r\geq 3n+2$. Suppose that:\\ \begin {inparaenum}
\item [(i)] for every $0$--dimensional, curvilinear scheme $\gamma$ of length 3 contained in the smooth locus of $X$, then $X$ is special along $\gamma$;\\
\item [(ii)]  $X$ is $2$--osculating regular.
\end{inparaenum}

Then $X$ is $2$--defective. 
\end{thm}

\begin{proof} First of all, after may be generically projecting down $X$ to $\PP^{3n+2}$, we may and will assume $r=3n+2$.

By the hypothesis (i) and Proposition \ref {prop:equal}, $X$ possesses, in the neighborhood of any smooth point, a family of dimension $3(n-1)$ of $\gamma_{1,5}$'s with the property that given a general curvilinear, 0--dimensional scheme $\gamma$ of length 3 on $X$, there is a unique $\gamma_{1,5}$ on $X$ containing  $\gamma$. Fix, as usual, a chart ${\bf x}(u_1,\ldots,u_n)$ of $X$ around a general point $p$ of $X$. We may require that the $u_1$--curves defined by $u_i={\rm const.}$, for $2\leq i\leq n$, are $\gamma_{1,5}$'s, and they are actually general $\gamma_{1,5}$'s. In particular, let us fix $u_i=0$, for $2\leq i\leq n$, so that we have the $\gamma_{1,5}$ curve $C$ parametrized by ${\bf x}(u_1,0\ldots,0)$. Then by Theorem \ref {thm:gamma}, we have that the $3n+3$ vectors 
\begin{equation}\label{eq:vv}
{\bf x}, {\bf x}_1, \ldots, {\bf x}_n,{\bf x}_{11}, \ldots, {\bf x}_{1n}, {\bf x}_{111}, \ldots, {\bf x}_{11n}, {\bf x}_{1111}, {\bf x}_{11111}
\end{equation}
are linearly dependent. However, by Remark \ref {rem:lap}, the first $3n+1$ vectors are linearly independent because of the hypothesis (ii). We denote by $\Pi$ the linear span of the points with coordinate vectors in \eqref {eq:vv} and by $\pi$ the linear span of the points with coordinate vectors the first $3n+1$ vectors in \eqref {eq:vv}. Then we have
\[
3n=\dim (\pi)\leq \dim(\Pi)\leq 3n+1.
\]
We will assume that $\dim(\Pi)=3n+1$, the proof running  in the same way (and in fact it is easier) if $\dim(\Pi)=3n$, i.e., if $\Pi=\pi$. Then we may assume that $\Pi$ is spanned by the points with coordinate vectors the first $3n+2$ vectors in \eqref {eq:vv} and ${\bf x}_{11111}$ depends on them. 

We note that $\Pi$ a priori depends on the variable $u_1$ which parametrizes the curve $\gamma_{1,5}$. However we make the following claim:

\begin {claim}\label{cl:main} The space $\Pi$ is constant with respect to $u_1$. 
\end{claim}

\begin{proof}[Proof of Claim \ref {cl:main}]

To prove the claim it suffices to prove that the derivatives with respect to $u_1$ of the first $3n+2$ vectors in \eqref {eq:vv} belong to span of the same vectors. This is clear for the derivatives of the vectors 
\[
{\bf x}, {\bf x}_1, \ldots, {\bf x}_n,{\bf x}_{11}, \ldots, {\bf x}_{1n}, {\bf x}_{111}, {\bf x}_{1111}
\]
and we have to prove it for the derivatives of the vectors ${\bf x}_{11h}$, for $2\leq h\leq n$. Namely we have to prove that the vectors ${\bf x}_{111h}$, for $2\leq h\leq n$, depend on the first $3n+2$ vectors in \eqref {eq:vv}. 

To see this, consider the determinant $D$ of the square matrix of size $3n+3$ whose columns are the vectors ${\bf x}, {\bf x}_1, \ldots, {\bf x}_n$ plus the vectors in \eqref {eq:vectors}. The determinant of this matrix is a polynomial in the variables $\lambda_i, \mu_i$, for $1\leq i\leq n$, which is identically zero, hence all the coefficients of its monomials have to be  zero. Let us consider the coefficient of the monomial $\lambda_1^{3n+7}\mu_2$ in $D$. To ease notation we denote by $S$ the ordered string of vectors ${\bf x}, {\bf x}_1, \ldots, {\bf x}_n,{\bf x}_{11}, \ldots, {\bf x}_{1n}$. The coefficient in question, which is zero, is a sum of determinants, some of which are clearly zero because two columns are equal. Writing the remaining terms, we get the relation
\begin{equation}\label{eq:part}
|S \,\, {\bf x}_{1111}\, {\bf x}_{111} \ldots {\bf x}_{11n}\, {\bf x}_{1112}|  
+2 \sum_{h=2}^n |S\,\, {\bf x}_{1111}\, {\bf x}_{111} \ldots {\bf x}_{11h-1}\,{\bf x}_{2h}\,{\bf x}_{11h+1}\ldots {\bf x}_{11n} {\bf x}_{11111} |\equiv 0.
\end{equation}
Next we compute the coefficient in $D$ of the monomial $\lambda_1^{3n+6}\lambda_2\mu_1$. Again, to ease notation we denote by $S'$ the ordered string of vectors ${\bf x}, {\bf x}_1, \ldots, {\bf x}_n$. 
As above we get the relation
\[
\begin{split}
&\qquad\qquad\qquad 0\equiv  |S \,\, {\bf x}_{1111}\, {\bf x}_{111} \ldots {\bf x}_{11n}\, {\bf x}_{1112}| +|S\,\, {\bf x}_{1112}\,{\bf x}_{111} \ldots {\bf x}_{11n}\, {\bf x}_{1111}|+\\
&+ 2\sum_{h=2}^n |S'\,\, {\bf x}_{11}, \ldots, {\bf x}_{1h-1}\, {\bf x}_{2h}\, {\bf x}_{1h+1}\ldots {\bf x}_{1n}\,
{\bf x}_{1111}\, {\bf x}_{111} \ldots {\bf x}_{11h-1}\,{\bf x}_{1h}\,{\bf x}_{11h+1}\ldots {\bf x}_{11n} {\bf x}_{11111} |.
\end{split}
\]
Now notice the the first and second term of this sum are opposite to each other, so they cancel out. In the  sum we have, for every $h\in \{2,\ldots, n\}$, that 
\[
\begin{split}
&|S'\,\, {\bf x}_{11}, \ldots, {\bf x}_{1h-1}\, {\bf x}_{2h}\, {\bf x}_{1h+1}\ldots {\bf x}_{1n}\,
{\bf x}_{1111}\, {\bf x}_{111} \ldots {\bf x}_{11h-1}\,{\bf x}_{1h}\,{\bf x}_{11h+1}\ldots {\bf x}_{11n} {\bf x}_{11111} |=\\
& \qquad\qquad\qquad=-|S\,\, {\bf x}_{1111}\, {\bf x}_{111} \ldots {\bf x}_{11h-1}\,{\bf x}_{2h}\,{\bf x}_{11h+1}\ldots {\bf x}_{11n}\, {\bf x}_{11111}|.
\end{split}
\]
So the summation in \eqref {eq:part} vanishes and therefore also $|S \,\, {\bf x}_{1111}\, {\bf x}_{111} \ldots {\bf x}_{11n}\, {\bf x}_{1112}| \equiv 0$. This proves that ${\bf x}_{1112}$  depends on the first $3n+2$ vectors in \eqref {eq:vv}. In a similar manner we see that also the vectors ${\bf x}_{111h}$, for $3\leq h\leq n$, depend on the first $3n+2$ vectors in \eqref {eq:vv}, and this proves the claim.\end{proof} 

Now we remark that the  $\gamma_{1,5}$ curve $C$ parametrized by ${\bf x}(u_1,0\ldots,0)$, which is a general $\gamma_{1,5}$ on $X$, spans the subspace $\Pi$ of dimension $3n+1$.

Moreover $\Pi$ contains also the tangent space to $X$ at the general point of $C$. In fact the point $p$ corresponding to $(u_1,\ldots, u_n)={\bf 0}$ is a general point of $X$ and $T_{X,p}=\langle {\bf x}, {\bf x}_1, \ldots, {\bf x}_n\rangle$. The assertion follows since the derivatives of the vectors ${\bf x}, {\bf x}_1, \ldots, {\bf x}_n$ with respect to $u_1$ still depend on the first $3n+2$ vectors in \eqref {eq:vv}, hence they give points belonging to $\Pi$.  

To finish the proof, showing that $X$ is $2$--secant defective, we take a general point $p\in X$ and two general points $q$, $r$ of $X$ sufficiently close to $X$ in the natural topology. Then there is a $\gamma_{1,5}$ curve $\Gamma$ containing $p,q,r$, which by the generality of the points $p,q,r$ is a general $\gamma_{1,5}$ on $X$. The span of $\Gamma$ is a linear space of dimension $3n+1$ and it contains the tangent spaces $T_{X,p}, T_{X,q}, T_{X,r}$. Hence 
\[
\dim \Big ( \langle T_{X,p}, T_{X,q}, T_{X,r}\rangle \Big )\leq 3n+1
\]
and $X$ is $2$--secant defective by the classical Terracini Lemma \ref {thm:lemma terra}.
\end{proof}

\printindex

\end{document}